\documentclass{amsart}
\usepackage{latexsym,amssymb,amsmath,amsthm,color,url,bbm}
\usepackage[comma]{natbib}
\usepackage{graphicx,pstricks}
\allowdisplaybreaks 
\numberwithin{equation}{section}

\newtheorem{Lemma}{Lemma}[section]

\newtheorem{Theorem}[Lemma]{Theorem}
\newtheorem{Proposition}[Lemma]{Proposition}
\newtheorem{Corollary}[Lemma]{Corollary}

\theoremstyle{definition}
\newtheorem{Remark}{Remark}[section]
\theoremstyle{definition}
\newtheorem{Example}{Example}[section]

\def\Pick{\text{Pickandsish}_{k,n}}
\def\Hill{\text{Hill}_{k,n}}
\def\Hillish{\text{Hillish}_{k,n}}
\def\Kendall{ \rho_{\tau}(k,n)}

\def\cinP{\stackrel{\text{P}}{\to}}
\def\cnvg{\stackrel{v}{\to}}

\def\bX{\boldsymbol X}
\def\bY{\boldsymbol Y}

\def\bone{\boldsymbol 1}

\def\bzero{\boldsymbol 0}

\def\binfty{\boldsymbol \infty}

\def\E{\mathbb{E}}
\def\P{\mathbf{P}}

\def\bM{\mathbb{M}}
\def\R{\mathbb{R}}

\def\Hinv{H^{\leftarrow}}
\def\tHinv{\tilde{H}^{\leftarrow}}
\def\inv{^{\leftarrow}}

\def\Enb{\mathbb{E}_{\sqcap}}

\begin{document}
\bibliographystyle{plainnat}
\date{\today}
\title[Detection: CEV models]{Detecting a conditional extreme value model
}
\author[B.\ Das]{Bikramjit\ Das}

\address{Bikramjit\ Das\\
School of Operations Research and Information Engineering\\
Cornell University \\
Ithaca, NY 14853} \email{bd72@cornell.edu}

\author[S.I.\ Resnick]{Sidney I.\ Resnick}
\address{Sidney Resnick\\
School of Operations Research and Information Engineering\\
Cornell University \\
Ithaca, NY 14853} \email{sir1@cornell.edu}
\thanks{B. Das and S. Resnick were partially supported by ARO Contract W911NF-07-1-0078
at Cornell University. }

\begin{abstract} In classical extreme value theory probabilities of extreme events are estimated assuming all the components
of a random vector to be in a domain of attraction of an extreme value distribution. In
contrast, the conditional extreme value model assumes a
domain of attraction condition on a sub-collection of the components of a multivariate random
    vector. This model has been studied in
    \cite{heffernan:tawn:2004,heffernan:resnick:2007,das:resnick:2008a}.
    In this paper we propose three statistics  which act as tools to
    detect  this model in a bivariate set-up. In addition, the
    proposed statistics  also help to distinguish between two forms
    of the limit measure that is obtained in the model.
\end{abstract}

\maketitle

\section{Introduction}\label{sec:intro}

 Extreme value theory (henceforth abbreviated EVT) is used to model
 environmental, financial and internet traffic data.  Multivariate
 extreme value theory  assumes an extreme-valued domain of attraction
 condition on the joint distribution of a random vector which after a
 suitable standardization  relates multivariate extreme value theory
 to regular variation on  $\R_{+}^d$. Probabilities of critical risk
 events can be  estimated under the regularly varying assumption and
 then related back to the original co-ordinate system. Such estimates
 are a function of  the dependence structure among the variables in
 the multivariate model. There is a substantial literature on the
 behavior of multivariate extreme-value theory both in the presence of
 asymptotic dependence and asymptotic independence
 \citep{coles:tawn:1991,ledford:tawn:1996,ledford:tawn:1997,ledford:tawn:1998,
dehaan:deronde:1998,resnick:2002a,maulik:resnick:2005,dehaan:ferreira:2006}.

 A \emph{conditional extreme value\/} (CEV) model was proposed in \cite{heffernan:tawn:2004}, where
 multivariate distributions were approximated by conditioning on one
 of the components being in an extreme-value domain.
The authors showed that this approach incorporated a variety of  examples of different types of asymptotic dependence and
 asymptotic independence. Their statistical ideas were given a more mathematical framework by
 \cite{heffernan:resnick:2007} after slight changes in  assumptions
 to make the theory more probabilistically viable. A further study (\cite{das:resnick:2008a}) revealed that assuming a conditional extreme value model holds no matter what conditioning variable is chosen is equivalent to assuming a multivariate extreme-value domain of attraction on the entire random vector.
  It is known that in the presence of asymptotic independence, the
  limit measure in the multivariate EVT set up has an empty interior;
  in other words, the limit measure concentrates on the boundary of
  the state space. In such a case an additional assumption of a CEV
  model provides more insight into the dependence structure. The CEV
  model also provides a way for modeling multivariate data assuming a
  subset  rather than the entire vector to be  extreme-valued. In this
  paper we suggest situations where a CEV model can be used and
  suggest techniques to detect the model and in the process also
  detect characteristics of the limit measure for the model. The
  methodologies are suggested for a bivariate data set.

 Section \ref{sec:intro}  provides an introduction and review of the model. Section \ref{sec:product} deals with the detection of  a conditional extreme value model. It has been shown in \cite{das:resnick:2008a} that the CEV model can be standardized to regular variation on a special cone if and only if the limit measure involved is not a product. In case of a product measure in the limit, we need to estimate fewer parameters and calculating probabilities is also simpler. \cite{fougeres:soulier:2008} suggests some  estimates for parameters and normalizing constants in the two different cases (product and non-product limit measures). Hence it is  important to  know whether we are in the product case or not. In Section \ref{sec:product} we propose three statistics whose behavior can first of all indicate the appropriateness of the CEV model and secondly indicate whether the limit measure is a product or not.  Section \ref{sec:data} is dedicated to applying our techniques to some simulated and real data coming from Internet traffic studies.

 \subsection{Preliminaries: The CEV model and some related results}\label{subsec:prelim}
 We discuss the model and preliminary results in this subsection. Refer to \cite{heffernan:resnick:2007} and \cite{das:resnick:2008a} for further discussion on conditional extreme value models.

  A univariate extreme value distribution $G_{\gamma}$, for $\gamma \in
 \R$, is defined as
 \begin{align}\label{eqn:Ggamma}
 G_{\gamma}(x)= \exp\{-(1+\gamma x)^{-1/\gamma}\},\qquad x \in \E^{(\gamma)},
 \end{align}
  where
 $\E^{(\gamma)} = \{ x \in \R : 1+ \gamma x > 0\}$. For $\gamma = 0$, the distribution function is interpreted as
 $G_{0}(x)= e^{-e^{-y}},\, y \in \E^{(0)}:=\mathbb{R}$.
 We let $\overline{\E}^{(\gamma)}$ be the right closure of
 $\E^{(\gamma)}$; that is,
 \begin{equation*}
 \overline{\E}^{(\gamma)}  = \begin{cases} (-\frac{1}{\gamma}, \infty]
                                         &  \gamma  >  0,\\
                                         (-\infty, \infty]   & \gamma  = 0,\\
                                         (- \infty, -\frac{1}{\gamma}] &  \gamma < 0.
 \end{cases}
 \end{equation*}

   Let $(X,Y)$ be a  random vector in $\R^2$.
For a CEV model we make the following assumptions:\begin{enumerate}
 \item
The distribution $F(x):=P[Y\leq x]$ is in the domain of attraction
of an extreme value distribution, $G_{\gamma}$, written $F\in
D(G_\gamma)$.
This means there exist
functions $a(t)>0$, $b(t) \in \R$ such that, as $ t \to \infty$,
\begin{align}
  t(1-F(a(t)y + b(t))) = t\P \biggl(\frac{Y-b(t)}{a(t)} > y \biggr)
  \to (1 + \gamma
y)^{-1/\gamma}, \qquad  1+\gamma y > 0, \label{eqn:doa}
\end{align} for $y \in \E^{(\gamma)}.$

\item There exist functions $\alpha(t) > 0$ and
$\beta(t) \in \R$ and a non-null Radon measure $\mu$ on Borel
subsets of $[-\infty,\infty] \times \overline{\E}^{(\gamma)} $ such
that for each $ y \in \E^{(\gamma)}$,
\begin{align}
  & [a] \quad t\P\biggl(\frac{X- \beta(t)}{\alpha(t)} \le x, \frac{Y - b(t)}{a(t)} > y\biggr)  \to \mu([-\infty,x] \times(y,\infty]), \quad \text{ as } t \to \infty, \label{eqn:cond2}\\
   \intertext{\qquad \qquad \quad \quad {for $(x,y)$ continuity points of the limit,
   }}
  & [b] \quad \mu([-\infty,x]\times(y,\infty]) \quad \text{ is not a degenerate distribution in $x$}, \label{eqn:nondegx}\\
  & [c] \quad  \mu([-\infty,x]\times(y,\infty]) < \infty.
  \label{eqn:munotinf}\\
  & [d] \quad H(x):= \mu([-\infty,x]\times(0 ,\infty]) \text{ is a probability
  distribution.} \label{eqn:Hispdf}
   \end{align}
 \end{enumerate}
 A non-null
 Radon measure $\mu(\cdot)$ is said to satisfy the \emph{conditional non-degeneracy
 conditions} if both of \eqref{eqn:nondegx} and \eqref{eqn:munotinf}
hold.
 We  say that $(X,Y)$  follows a \emph{conditional extreme value
 model} (or CEV model) if conditions (1) and (2) above are satisfied. The
reason for this name is that, assuming $(x,0)$ to be a continuity point of $\mu$, \eqref{eqn:cond2}, \eqref{eqn:nondegx}
and
 \eqref{eqn:munotinf} imply that
 \begin{align}
 \P\biggl(\frac{X-\beta(t)}{\alpha(t)} \le x \bigg| Y > b(t)\biggr) \to H(x) =
 \mu([-\infty,x]\times(0 ,\infty]) \qquad (t \to \infty).
 \label{eqn:constd}
 \end{align}
 Note that \eqref{eqn:cond2} can be viewed in terms of vague
 convergence of measures in $\mathbb{M}_{+}\bigl(
[-\infty,\infty] \times
 \overline{\E}^{(\gamma)}
\bigr)$, the space of Radon measures on $[-\infty,\infty] \times
 \overline{\E}^{(\gamma)}
$.

In this model the transformation $Y \mapsto Y^*=b\inv(Y)$  standardizes the $Y$-variable, i.e., we can assume  $a^*(t)=t,b^*(t)=0$. Hence a reformulation of
\eqref{eqn:cond2} leads to
\begin{align}
 t\P\biggl(\frac{X- \beta(t)}{\alpha(t)} \le x, \frac{Y^*}{t} > y\biggr)  & \to \mu^{*}([-\infty,x] \times(y,\infty]), \quad \text{ as } t \to \infty,
\intertext{for $(x,y)$ continuity points of $\mu^{*}$ where,}
 \mu^{*}([-\infty,x] \times(y,\infty])  & =
\begin{cases}
\mu \bigl( [-\infty,x]\times (\frac{y^\gamma -1}{\gamma},\infty]
\bigr),& \text{ if } \gamma \neq 0,\\
\mu \bigl( [-\infty,x]\times (\log y,\infty]
\bigr),& \text{ if } \gamma = 0. \label{eqn:stdizemu}\\
\end{cases}
\end{align}

We also know \citep[Proposition 1]{heffernan:resnick:2007} that  the following
variational property holds: there exists functions $\psi_1(\cdot),
\psi_2(\cdot)$ such that for all $c>0$,
\begin{align}
&\lim\limits_{t \to \infty} \frac{\alpha(tc)}{\alpha(t)}  =
\psi_1(c), & &\lim\limits_{t \to \infty}
\frac{\beta(tc)-\beta(t)}{\alpha(t)}  = \psi_2(c).
\label{eqn:condpsi1psi2}
\end{align}
This implies that $\psi_1(c) = c^{\rho}$ for some $\rho \in \R$
\cite[Theorem B.1.3]{dehaan:ferreira:2006}. Also, $\psi_2$ can be either
0 or $\psi_2(c) = k \frac{c^{\rho}-1}{\rho}$ for some $k \neq 0$
\citep[Theorem B.2.1]{dehaan:ferreira:2006}.

\begin{Remark} The CEV model primarily differs from the multivariate extreme value model
in the domain of attraction condition. \cite{das:resnick:2008a}
provides conditions under which a CEV model can be extended to
multivariate extreme value model. Under the multivariate extreme
value model, each of the variables can be standardized so that we
have a multivariate regular variation on the cone $[\bzero,\binfty]
\setminus \{\bzero\}$; see  \cite{dehaan:resnick:1977}
and Chapter 6 of \citet{dehaan:ferreira:2006}. The
conditional extreme value model can  be standardized if and
only if the limit measure $\mu$ in \eqref{eqn:cond2} is  not a
product measure \citep{das:resnick:2008a}.
When both $X$ and $Y$ are standardized,
 we can characterize the limit measure
in terms of all Radon measures (finite and infinite) on $[0,1)$.
Though theoretically elegant, performing
standardization in practice is not an easy task.

Thus, it is important to know when
 the limit is a product measure.
 A product
measure in the limit precludes standardization of both the variables
\citep{heffernan:resnick:2007} and   means that  we do not have a
multivariate extreme value model \citep{das:resnick:2008a}. However, a
product measure
makes the estimation of certain parameters and probabilities easier.
For instance, in the product case $(\psi_1,\psi_2)\equiv (1,0)$ so
$\rho=0$ (see \eqref{eqn:condpsi1psi2})
but without the property that $\mu $ is a product, $\rho $ has to
be estimated. Furthermore, the limit being a
product measure can be considered as a form of \emph{asymptotic
  independence\/} in the CEV model, which can be probabilistically
useful \citep{maulik:resnick:rootzen:2002}.
\end{Remark}

\subsection{Appropriateness of the CEV model.} \label{subsec:whycevm}
Multivariate
extreme value theory provides a rich literature on estimation of
probabilities of extreme regions containing few or no data
points in the sample. The multivariate theory assumes that each variable is
marginally in an extreme value domain of attraction. However, this might not
be the right assumption for all data sets. We encounter data where one or some but not all the variables can be
assumed to be in an extreme value domain; see Section
\ref{subsec:Internet}. The CEV model is a candidate
 model in such  cases.

 Another circumstance where the CEV model can be helpful is
if one has a multivariate extreme value model with limit measure $\nu$
possessing \emph{asymptotic independence}. This means that in the
standardized model, the limit measure, $\nu^{*}(\cdot)$, concentrates
on the axes through $\{\bzero\}$ and $\nu^*((\bzero,\binfty])=0$. So an
estimate of the probability of a region where both variables are big
 will turn out to be zero which may be a useless and misleading
estimate. In such a circumstance, finer estimates
 can be obtained using either
\emph{hidden regular variation} \citep{maulik:resnick:2005}
or the CEV model. Both methods provide a non-zero limit
measure by using normalization functions which are of different order
from the multivariate EV model.  The relationship between the
multivariate EV model and the CEV model and the respective normalizing
functions considered in \cite{das:resnick:2008a}.

So, how do we decide if the CEV model is appropriate for multivariate
data?
\begin{enumerate}
\item Start by  checking whether any
of the marginal variables belongs to the domain of attraction of an
extreme value distribution. An informal way to do this is through plots of the
estimators of the extreme value parameter $\gamma$ (Pickands plot,
Moment estimator plot, etc). If the plot attains stability in some
range it is reasonable to assume an extreme-value model.
More formal methods
for testing membership in a domain of attraction using quantile
 and distribution
functions are discussed in \cite{dehaan:ferreira:2006}, Chapter
5.2. The special case of a heavy-tailed random variable can be
detected using the QQ plot, plotting the theoretical quantiles of the
exponential distribution versus the logarithm of the sorted data and
checking for linearity in the high values of the data. This is
reviewed in \cite{resnickbook:2007} and \cite{das:resnick:2008}.

\item If some, but not all, marginal variables are in a domain of attraction,
  proceed to see if the data is consistent with the CEV model. See
  Section \ref{sec:product}.

\item  If all variables are in some extreme value domain,  check if
  the multivariate extreme value model is appropriate and if asymptotic
independence is present. One way to do this is by checking whether both maximum
and minimum of the standardized variables have distributions with
regularly varying tails \citep{coles:heffernan:tawn:1999,resnick:2002a}. If the EV model is
appropriate and
asymptotic independence is absent, the CEV model does not
provide  any more information than the EV model. On the other hand if
asymptotic independence is present,  the CEV model, if detected,  provides supplementary information about
the joint behavior of the variables away from at least one of the
axes.
\end{enumerate}

\section{Three estimators for detecting the CEV model }\label{sec:product}
Let $(X_1, Y_1), \ldots, (X_n, Y_n)$ be a bivariate random sample. In
this section we propose three statistics to detect whether our sample
is consistent with the CEV model under the assumption that at least
one of the variables is in an extreme-value domain, and without loss
of generality we assume $Y$ to be that variable. Our statistics have a
consistency property which allows
  detection of a product form for the limit measure.

Assume $(X_1, Y_1), \ldots, (X_n, Y_n)$
 is iid  from  a CEV model as defined in Section \ref{subsec:prelim}.
We first formulate a consequence of \eqref{eqn:cond2} which will
be convenient for our purpose. The following notations will be used.
$$
\begin{array}{llll}
Y_{(1)} \ge \ldots \ge Y_{(n)} & \text{The decreasing order
statistics of  $Y_1,\ldots,Y_n$.}\\[2mm]
X_i^{*}, ~ 1 \le i \le n & \text{The $X$-variable corresponding to
$Y_{(i)}$, also called the concomitant of $Y_{(i)}$.}\\ [2mm]
R_{i}^k= \sum\limits_{l=i}^k \bone_{\{X_l^{*} \le X_i^{*}\}}& \text{Rank of $X_i^{*}$ among
$X_1^{*},\ldots,X_{k}^{*}$. We write $R_i=R_i^k$.}\\ [2mm]
X_{1:k}^{*} \le X_{2:k}^{*} \le \ldots \leq X_{k:k}^{*} & \text{The
increasing order statistics of $X_1^{*},\ldots,X_k^{*}$.}
\end{array}
$$

\subsection{A consequence for empirical measures}
\label{subsec:ranks} When the CEV property holds,
a family of point processes of ranks of the sample converge vaguely to a Radon measure.
By transforming to ranks of the data, we presumably lose  efficiency since
only the relative ordering in the sample remains unchanged but
detection of the CEV property is  easier since we no longer need  to
estimate the various
parameters of the model. See \citet{dehaan:deronde:1998, dehaan:ferreira:2006,
 resnickbook:2007}.

The convergence statement \eqref{eqn:cond2} of the CEV model defined
in Section \ref{subsec:prelim} can be interpreted
 in terms of vague convergence of measures. In preparation  for the forthcoming result we recall some commonly used notation and concepts.  Let $\E^{*}$ be  a locally compact space
with a countable base (for example, a finite dimensional Euclidean
space). We denote by $\mathbb{M}_+(\E^{*})$, the non-negative Radon measures on
Borel subsets of $\E^{*}$.
If $\mu_n \in \mathbb{M}_+(\E^{*})$ for $n \geq 0$, then $\mu_n$ converge
vaguely to $\mu_0$ (written $\mu_n \stackrel{v}{\to} \mu_0$) if for
all bounded continuous functions $f$ with compact support we have
$$\int_{\E^{*}} fd\mu_n \to \int_{\E^{*}} fd\mu_0\quad (n \to \infty).$$
This concept allows us to write  \eqref{eqn:cond2}  as
\begin{align} \label{eqn:tovag}
 t \P\Bigg(\Big(\frac{X- \beta(t)}{\alpha(t)}, \frac{Y - b(t )}{a(t)}\Big) \in \cdot \Bigg) &
 \cnvg
\mu(\cdot), \quad \text{ as } t \to \infty
\end{align}
in $\bM_{+}([-\infty,\infty]\times\overline{\E}^{(\gamma)})$. Standard references
include \citet{kallenberg:1983, neveu:1977} and \citet[Chapter
3]{resnickbook:2008}.

 Recall the definition of $\mu^*$ in \eqref{eqn:stdizemu} and define
 the measure $L(\cdot) \in \bM_{+}([0,1]\times [1,\infty])$ by
\begin{equation}\label{eqn:almostCopula}
L([0,x]\times
(y,\infty])=\mu^{*}([-\infty,\Hinv(x)]\times [y,\infty]),
\quad (x,y)\in [0,1]\times (1,\infty].
\end{equation}
Applying a reciprocal transformation to the second coordinate,
$$T_0: (x,y) \mapsto (x,y^{-1})$$
converts $L$ into the copula $L\circ T_0^{-1}$.

\begin{Proposition}\label{prop:empconv}
Suppose $(X_1,Y_1),(X_2,Y_2),\ldots,(X_n,Y_n)$ are i.i.d. observations
from a CEV model which follows \eqref{eqn:doa}-\eqref{eqn:Hispdf} and
suppose $H$ is continuous. If $k=k(n)
\to \infty, n \to \infty$ with $k/n \to 0$, then in $\bM_+([0,1]\times [1,\infty])$
\begin{align*}
\frac{1}{k} \sum\limits_{i=1}^{k}
\epsilon_{(\frac{R_i}{k},\frac{k+1}{i})} (\cdot) \Rightarrow L(\cdot).
\end{align*}
\end{Proposition}

\begin{proof}
From \eqref{eqn:tovag} and
\cite[Theorem 5.3(ii)]{resnickbook:2007}, as $n,k \to \infty$ with $\frac
kn\to 0$,
\begin{align}
 \frac 1k \sum\limits_{i=1}^n \epsilon_{\big(\frac{X_i-
\beta (n/k)}{\alpha(n/k)} , \frac{Y_i - b( n/k )}{a( n/k)}\big)}
(\cdot) & \Rightarrow \mu(\cdot), \label{eqn:condnk}
\end{align}
in
$\bM_{+}([-\infty,\infty]\times\overline{\E}^{(\gamma)})$.  Recall $Y_{(1)}
\ge Y_{(2)} \ge \ldots \ge Y_{(n)} $ are the order statistics of
$Y_1,\ldots, Y_n$ in decreasing order and ordering the $Y$'s in
\eqref{eqn:condnk} allows us to write the equivalent statement
\begin{equation}\label{eqn:orderedVer}
  \frac 1k
\sum\limits_{i=1}^n \epsilon_{\big(\frac{X_i^*- \beta
(n/k)}{\alpha(n/k)} , \frac{Y_{(i)} - b( n/k )}{a( n/k)}\big)}
\Rightarrow \mu(\cdot).
\end{equation}
Define the measure $\nu_\gamma$ by
$$\nu_\gamma \bigl((y,\infty]\cap \overline{ \E}^{(\gamma)} \bigr)
=(1+\gamma y)^{-1/\gamma},\quad y\in \E^{(\gamma)},$$
and sometimes, here and elsewhere,
 we sloppily write  $\nu_{\gamma}(y,\infty]$.
Taking marginal convergence
in \eqref{eqn:condnk}, or using \eqref{eqn:doa}, we have
with
that
\begin{align}
\frac 1k \sum\limits_{i=1}^n \epsilon_{ \frac{Y_i - b( n/k) }{a( n/k)}}
(\cdot)  & \Rightarrow \nu_\gamma (\cdot),  \nonumber
\intertext{in $\bM_{+}(\overline{\E}^{(\gamma)}$.  Using an inversion technique
(\citet{resnick:starica:1995, dehaan:ferreira:2006}, \citet[page
82]{resnickbook:2007}), we get}
\frac{Y_{(\lceil (k+1)t \rceil)} -
b(n/k)}{a(n/k)} & \cinP \frac{t^{-\gamma}-1}{\gamma},
  \label{eqn:Ybnk}
\end{align}
in $D_l\bigl((0,\infty],\overline{\E}^{(\gamma)}\bigr)$, the class of
left continuous functions on $(0,\infty]$ with range
$\overline{\E}^{(\gamma)}$ and with finite right limits on $(0,\infty)$.  The convergence in \eqref{eqn:Ybnk} being to a
non-random function, we can append it to the convergence in
\eqref{eqn:orderedVer} to get the following
\cite[p.27]{billingsley:1968}:
\begin{align}
(\mu_n,x_n(t):=&\Bigg(
  \frac 1k
\sum\limits_{i=1}^n \epsilon_{\big(\frac{X_i^*- \beta
(n/k)}{\alpha(n/k)} , \frac{Y_{(i)} - b( n/k )}{a( n/k)}\big)}(\cdot)
 , \frac{Y_{(\lceil (k+1)t \rceil)} - b(n/k)}{a(n/k)} \Bigg)
\nonumber \\
 \Rightarrow &\Big(\mu(\cdot), \frac{t^{-\gamma}-1}{\gamma} \Big)=(\mu,x_\infty(t))
\label{eqn:cond2nk4}
\end{align}
on $\bM_{+}([-\infty,\infty]\times\overline{\E}^{(\gamma)})\times
D_l\bigl((0,\infty],\overline{\E}^{(\gamma)})\bigr)$.

Let $D_l^\downarrow\bigl((0,\infty],\overline{\E}^{(\gamma)})\bigr)$
be the subfamily of
$D_l\bigl((0,\infty],\overline{\E}^{(\gamma)})\bigr)$ consisting of
non-increasing functions and define
$$T_1:\mathbb{M}_+([-\infty,\infty]\times \overline{\E}^{(\gamma)}) \times
D_l^\downarrow \bigl((0,\infty],\overline{\E}^{(\gamma)})\bigr)
\mapsto
\mathbb{M}_+ \bigl([-\infty,\infty]\times (0,\infty]
\bigr)
$$
by
$$T_1(m,x(\cdot))=m^*$$
where
$$m^*([-\infty,x]\times (t,\infty])
=m([-\infty,x]\times (x(t^{-1}),\infty]),\quad x \in
[-\infty,\infty],\,t\in (0,\infty].
$$
This is an a.s. continuous map so apply this to \eqref{eqn:cond2nk4}
and
\begin{equation}\label{eqn:hideUgly}
T_1(\mu_n,x_n)\Rightarrow T_1(\mu,x_\infty).
\end{equation}
For $x \in
[-\infty,\infty]$ and $ y \in (0,\infty]$, the left side of
\eqref{eqn:hideUgly}
on the set $[-\infty,x]\times (y,\infty]$ is
$$
\mu_n\bigl(
[-\infty,x]\times (\frac{Y_{(\lceil (k+1)t \rceil)} -
  b(n/k)}{a(n/k)},\infty]\bigr)$$
and since
$$\frac{Y_{(i)} - b( n/k )}{a( n/k)}>
\frac{Y_{(\lceil (k+1)t \rceil)} -
  b(n/k)}{a(n/k)}
$$
iff
$$i<(k+1)y^{-1}\quad \text{ or } \quad  \frac{k+1}{i}>y,$$
the left side of \eqref{eqn:hideUgly} on the set $[-\infty,x]\times (y,\infty]$ is
$$
  \frac 1k
\sum\limits_{i=1}^n \epsilon_{\big(\frac{X_i^*- \beta
(n/k)}{\alpha(n/k)} , \frac{k+1}{i}\big)}([-\infty,x]\times (y,\infty]).$$
The right side of \eqref{eqn:hideUgly} on the set $[-\infty,x]\times
(y,\infty]$
is
$$\mu\bigl([-\infty,x]\times (\frac{y^\gamma-1}{\gamma},\infty]\bigr)
=\mu^*\bigl([-\infty,x]\times (y,\infty]\bigr),$$
so we conclude
\begin{equation}\label{eqn:sharpsharp}
  \frac 1k
\sum\limits_{i=1}^n \epsilon_{\big(\frac{X_i^*- \beta
(n/k)}{\alpha(n/k)} , \frac{k+1}{i}\big)} \Rightarrow \mu^*
\end{equation}
in $\mathbb{M}_+\bigl([-\infty,\infty]\times (0,\infty]\bigr).$
Recall $\mu^*(\cdot)$ was  defined
in \eqref{eqn:stdizemu}.

Now assuming  $(x,1)$ is a continuity point of $\mu^*$ we have
\begin{align}
H_n(x) :=&
\frac {1}{k} \sum\limits_{i=1}^k \epsilon_{\big(\frac{X_i^*-
\beta (n/k)}{\alpha(n/k)} , \frac {k+1}{i}\big)} ([-\infty,x]
\times(1,\infty]) \nonumber \\
 \Rightarrow &\mu^*([-\infty,x] \times(1,\infty])
 =: H(x),\label{eqn:nonumber}
\end{align}
or in the topology of weak convergence on $PM[-\infty,\infty]$, the
probability measures on
$[-\infty,\infty] $,
$$H_n \Rightarrow H.$$
Define a map $T_2$ on $\bM_{+}([-\infty,\infty]\times [1,\infty])\times
PM[-\infty,\infty] $ by
$$T_2(m,G) = m^\#$$
where
$$
m^\#([0,z]\times (y,\infty])=m([0,G(z)]\times (y,\infty])
$$
or, for $f \in C([0,1]\times [1,\infty])$,
$$m^\# (f)=\iint f(G(x),y) m(dx,dy).$$
This map is continuous at $(m,G)$ provided $G$ is continuous. To see
this, let $f $ be continuous on $[0,1]\times [1,\infty]$ and suppose
$G_n\Rightarrow G$ and $m_n \stackrel{v}{\to} m$. Then
\begin{align*}
\Bigl|\iint f(G_n(x),y)&m_n(dx,dy)
-\iint f(G(x),y)m_n(dx,dy)\Bigr|\\
\leq & \Bigl|\iint f(G_n(x),y)m_n(dx,dy) -
\iint f(G(x),y)m_n(dx,dy)\Bigr|\\
&\qquad +\Bigl|
\iint f(G(x),y)m_n(dx,dy)-\iint f(G(x),y)m(dx,dy)\Bigr|\\
=&I+II.
\end{align*}
For $I$, convergence to $0$ follows by uniform continuity of $f$ and
the fact that $G_n(x)\to G(x)$ uniformly in $x$. To verify $II \to 0$,
it suffices to note that  $f(G(x),y)$ is continuous with compact
support $[-\infty,\infty]\times [1,\infty]$ and then use $m_n
\stackrel{v}{\to} m$.

Combine \eqref{eqn:sharpsharp} and \eqref{eqn:nonumber} to get
\begin{equation}\label{eqn:sharpsharpsharp}
 \Bigl(
 \frac 1k
\sum\limits_{i=1}^k \epsilon_{\big(\frac{X_i^*- \beta
(n/k)}{\alpha(n/k)} , \frac{k+1}{i}\big)} ,
H_n \Bigr)
\Rightarrow (\mu^*,H)
\end{equation}
in $\mathbb{M}_+\bigl([-\infty,\infty]\times [1,\infty]\bigr) \times PM[-\infty,\infty]$.
Apply the transformation $T_2$ discussed in the previous
paragraph. The limit at $[0,x]\times (y,\infty]$ is
$\mu^*\bigl([-\infty,H^\leftarrow (x)]\times
(y,\infty]\bigr)=L[0,x]\times (y,\infty]$. The
converging sequence can be written as
$$\frac 1k \sum_{i=1}^k \epsilon_{
\Bigl(H_n
\bigl( \frac{X_i^*  -\beta(n/k)}{\alpha(n/k) }\bigr)   , \frac{k+1}{i }\Bigr)
}.$$
Finally observe
$$H_n
\bigl( \frac{X_i^*  -\beta(n/k)}{\alpha(n/k) }\bigr)
=\frac 1k  \sum_{l=1}^k 1_{[\frac{X_l^* -\beta(n/k)}{\alpha(n/k)} \leq
\frac{X_i^* -\beta(n/k)}{\alpha(n/k)}]} =\frac{R_i}{k}.$$
The result follows.
\end{proof}

We propose  three statistics that can be used to detect whether or not a CEV model is appropriate, and if so, whether the model has a product measure in the limit.
 \subsection{The Hillish statistic, $\Hillish$}
The Hill estimator (\cite{hill:1975,mason:1982,dehaan:ferreira:2006,resnickbook:2007}) is a popular choice
for estimating the tail parameter $\alpha$ of a heavy-tailed
distribution. We say that a distribution function $F$ on $\R$ is
heavy-tailed with tail parameter $\alpha > 0$ if
      \begin{align} \label{eqn:defht}
      \lim\limits_{t\to \infty} \frac{1-F(tx)}{1-F(t)} = x^{-\alpha} \qquad \text{for} ~~~ x>0.
      \end{align}
If $Z_1,Z_2,\ldots,Z_n$ are i.i.d from this distribution
$F$ and $Z_{(1)} \ge Z_{(2)} \ge \ldots \ge Z_{(n)}$ are the orders
statistics of the sample in decreasing order,
then the Hill estimator defined as
$$
\Hill =
\frac{1}{k}\sum\limits_{j=1}^{k} \log
\frac{Z_{(i)}}{Z_{(k+1)}}$$
is a weakly consistent estimator of $\frac{1}{\alpha}$  as $k,n\to
\infty, k/n \to 0$.  One way to obtain the consistency is to integrate
the tail empirical measure and using its consistency. See
\citet{resnick:starica:1995} or
\citet[p. 81]{resnickbook:2007}.

The Hillish statistic, based on the ranks of the sample, converges weakly to a constant limit under the CEV model.  The name is derived from the similarity of proof of this convergence with that of the weak consistency of the Hill estimator. Using the notation defined just prior to Section 2.1, and assuming   $(\bX,\bY):=\{(X_1,Y_1),(X_2,Y_2),\ldots,(X_n,Y_n)\}$, the Hillish statistic for $(\bX,\bY)$ is defined as
 \begin{align}\label{def:bihill}
 \Hillish(\bX,\bY) := \frac{1}{k} \sum\limits_{j=1}^{k} \log \frac{k}{R_{j}}\log
\frac{k}{j}.
 \end{align}
 The following proposition provides a convergence result of the Hillish statistic under conditions on $k$.
\begin{Proposition}\label{prop:hillishconv}
Suppose $(X_1,Y_1),(X_2,Y_2),\ldots,(X_n,Y_n)$ are i.i.d. observations
from a CEV model which follows \eqref{eqn:doa}-\eqref{eqn:Hispdf} and suppose $H$ as defined in \eqref{eqn:Hispdf}
is continuous. Assume that $k=k(n) \to \infty, \, n
\to \infty$ and $k/n\to 0$. Then
\begin{align} \label{eqn:hilllim}
 \Hillish \cinP \int\limits_{1}^{\infty}
\int\limits_{1}^{\infty}\mu^{*}([-\infty,\Hinv(\frac
1x)]\times(y,\infty]) \frac{dx}{x}\frac{dy}{y} = : I_{\mu^*}.
\end{align}
\end{Proposition}
\begin{proof}
Proposition \ref{prop:empconv} yields
\begin{align}\label{eqn:prop2.1inv}
 \frac 1k \sum\limits_{i=1}^n \epsilon_{\big(\frac{R_i}k , \frac
{k+1}i \big)} (\cdot)  & \Rightarrow
L(\cdot)
\end{align}
in $\bM_{+}\bigl([0,1]\times [1,\infty]\bigr)$.
Rewrite \eqref{eqn:prop2.1inv} for $x\ge1,y>1$ as
 \begin{align}
    \mu^*_n([x,\infty] \times(y,\infty])&:=  \frac
1k \sum\limits_{i=1}^n \epsilon_{\big(\frac k{R_i} , \frac
{k+1}i \big)} ([x,\infty] \times(y,\infty]) \Rightarrow
\mu^*([-\infty,H\inv (1/x)] \times(y,\infty]).
\label{eqn:cond2nk8}            \end{align}
 Observe that
  \begin{align}
 I_n & :=  \int\limits_1^{\infty}\int\limits_1^{\infty}
 \mu^*_n ([x,\infty] \times(y,\infty]) \frac{dx}{x}
\frac{dy}{y}
 = \frac 1 k \int\limits_1^{\infty}\int\limits_1^{\infty}
\sum\limits_{i=1}^n \epsilon_{\big(\frac k{R_i} , \frac {k+1}i
\big)} ([x,\infty] \times(y,\infty]) \frac{dx}{x}
\frac{dy}{y}\nonumber\\
& =  \frac 1k  \sum\limits_{i=1}^k \log \frac {k}{R_i} \log
\frac{k+1}{i}
 = \frac 1k  \sum\limits_{i=1}^k \log \frac {k}{R_i} (\log
\frac{k}{i} + \log \frac {k+1}{k})\nonumber\\
& =  \frac 1k  \sum\limits_{i=1}^k \log \frac {k}{R_i} \log
\frac{k}{i} + \big(\log \frac{k+1}{k}\big)  \frac 1k
\sum\limits_{i=1}^k \log
\frac {k}{R_i}
 = \Hillish + A_k \label{eqn:HkAk}
\end{align}
where  $A_k:= \big(\log\frac{k}{k+1}\big) \frac 1k
\sum\limits_{i=1}^k \log \frac k i \to 0\times 1=0$ as $k \to
\infty$. Hence if we show $$I_n \cinP \int\limits_{1}^{\infty}
\int\limits_{1}^{\infty}\mu^{*}([-\infty,\Hinv(\frac
1x)]\times(y,\infty]) \frac{dx}{x}\frac{dy}{y},$$ then we are done.
For $N$ finite we know that
\begin{align}
 \int\limits_1^N\int\limits_1^N
 \mu^*_n ([x,\infty] \times(y,\infty]) \frac{dx}{x}
\frac{dy}{y} \cinP  \int\limits_1^N\int\limits_1^N
 \mu^*([-\infty,H\inv(1/x)] \times(y,\infty]) \frac{dx}{x}
\frac{dy}{y}, \label{eqn:munstarmu}
\end{align} since \eqref{eqn:cond2nk8} implies that the
integrand converges in probability and we can use Pratt's Lemma
\citep[page 164]{resnick:1999book} for the convergence of the integral.
Note that  as $N \to \infty$ the right hand side in equation
\eqref{eqn:munstarmu} converges to
$\int\limits_{1}^{\infty}
\int\limits_{1}^{\infty}\mu^{*}([-\infty,\Hinv(\frac
1x)]\times(y,\infty]) \frac{dx}{x}\frac{dy}{y}$.
So we need to see what happens outside the compact sets. In
particular if we can show that for any $\delta >0$,
\begin{align}
& \lim\limits_{N\to \infty} \limsup\limits_{n\to \infty}
\P\Bigg(\int\limits_N^{\infty}\int\limits_1^{\infty}
 \mu^*_n ([x,\infty] \times(y,\infty]) \frac{dx}{x}
\frac{dy}{y} > \delta\Bigg) = 0 \label{eqn:hillconv1}\\
\text {and}\qquad & \lim\limits_{N\to \infty} \limsup\limits_{n\to
\infty} \P\Bigg(\int\limits_1^{\infty}\int\limits_N^{\infty}
 \mu^*_n ([x,\infty] \times(y,\infty]) \frac{dx}{x}
\frac{dy}{y} > \delta\Bigg) = 0  \label{eqn:hillconv2}
\end{align}
then by a standard converging together theorem \citep[Theorem
3.5]{resnickbook:2007}, we are done. Observe that
\begin{align*}
0 & \le \int\limits_1^{\infty}\int\limits_N^{\infty}
 \mu^*_n ([x,\infty] \times(y,\infty]) \frac{dx}{x}
\frac{dy}{y}   =\frac{1}{k} \sum\limits_{j=1}^{k} \log
\frac{k}{R_{j}}(\log 1 \vee \log \frac{k}{jN})\\
& \le \sqrt{\frac{1}{k} \sum\limits_{j=1}^{k} \big(\log
\frac{k}{R_{j}}\big)^2 \frac{1}{k}\sum\limits_{j=1}^{k} \big(\log1
\vee \log \frac{k}{jN}\big)^2} \qquad \text{(Cauchy-Schwarz)}\\
& = B_{k_n} \times C_{k_n,N} \intertext{where} B_{k_n}^2 & =
\frac{1}{k}\sum\limits_{j=1}^{k} \big( \log \frac{k}{R_j}\big)^2 =
\frac{1}{k}\sum\limits_{j=1}^{k} \big( \log \frac{k}{j}\big)^2  \sim
\int\limits_0^1 (-\log x)^2 dx=2 \intertext{and} C_{k_n,N}^2 & =
\frac{1}{k}\sum\limits_{j=1}^{k} \big( 0 \vee \log
\frac{k}{jN}\big)^2 = \frac{1}{k}\sum\limits_{j \le  k/N} \big(
\log \frac{k}{jN}\big)^2\\
& = \frac{1}{N} \frac{1}{k/N} \sum\limits_{j=1}^{ k/N} \big( \log
\frac{k}{jN}\big)^2 \sim \frac 1N \int\limits_0^1 (-\log x)^2 dx=
\frac 1N \times 2.
\end{align*}
Hence
\begin{align*}
\limsup\limits_{n\to \infty}
\P\Bigl[\int\limits_1^{\infty}\int\limits_N^{\infty} &
 \mu^*_n ([x,\infty] \times(y,\infty]) \frac{dx}{x}
\frac{dy}{y}> \delta \Bigr]  \le \limsup\limits_{n\to \infty}
\P(B_{k_n}\times C_{k_n,N}>\delta)\\
\intertext{and applying Fatou's Lemma, this is bounded by}
& \le \P\Bigl[\sqrt{2\times\frac{2}{N}} > \delta \Bigr]
 \to 0 \qquad (N \to \infty).
 \end{align*}
 This shows \eqref{eqn:hillconv2} holds and similarly we can show
 \eqref{eqn:hillconv1} holds, and we are done.
\end{proof}
Suppose $(\bX,\bY):=\{(X_1,Y_1),(X_2,Y_2),\ldots,(X_n,Y_n)\}$ is a
sample from a CEV limit model with normalizing functions $\alpha,
\beta, a,b$ and variational functions $\psi_1, \psi_2$. Let the
standardized limit measure be $\mu^*$ as defined in
\eqref{eqn:stdizemu}. Also $H(x)=
\mu^*([-\infty,x]\times(1,\infty])$. Then $(-\bX,\bY)$ is also a
sample from a  CEV limit model but with normalizing functions
$\tilde{\alpha}=\alpha, \tilde{\beta}= -\beta, \tilde{a}=a
,\tilde{b}=b$ and variational functions $\tilde{\psi_1}=\psi_{1},
\tilde{\psi_2}=-\psi_{2}$. In this case the standardized limit measure
is $\tilde{\mu^*}$ and it is easy to check that for $x \in \R, \, y > 0$,
\begin{align}
\tilde{\mu^*}([-\infty,x]\times(y,\infty]) & = \mu^*([-x,\infty]\times(y,\infty]) . \label{eqn:hatmumu}
\end{align}

We also have for $x \in \R$
\begin{align}
\tilde{H}(x) := \tilde{\mu^*}([-\infty,x]\times(1,\infty])
            = {\mu^*}([-x,\infty]\times(1,\infty])
           = 1-H(-x).\label{eqn:hatHH}
\end{align}
Thus, for $0<p<1$, we have,
\begin{align}
\tHinv(p) = - \Hinv(1-p).   \label{eqn:hatHinvHinv}
\end{align}

The following proposition characterizes product measure in terms of limits of the Hillish statistic for both $(\bX,\bY)$ and $(-\bX,\bY)$.
\begin{Proposition} \label{prop:hillishprod} Under the conditions of
Proposition \ref{prop:hillishconv}, $\mu^*$ is a product measure if and only if both \begin{align*}
\Hillish(\bX,\bY) \cinP 1 \quad\text{and} \quad \Hillish(-\bX,\bY) \cinP 1.
\end{align*}
\end{Proposition}
\begin{proof}
Evaluating $I_{\mu^*}$, the limit of the Hillish statistic as proposed in Proposition \ref{prop:hillishconv},
leads us to the above results. Recall that we
assume $H$ is continuous. Define for any $c>0$, the family $H^{(c)}(\cdot)$ of distribution functions as follows:
\begin{align*}
H^{(c)}(x)  := c^{-1} \mu^*([-\infty,x]\times(c^{-1},\infty])
 = H(\psi_1(c) x +\psi_2(c))
\end{align*}
where $\psi_1,\psi_2$ are as defined in \eqref{eqn:condpsi1psi2} and
the second equality can be obtained by using $tc$
instead of $t$ in the CEV model property \eqref{eqn:cond2} \citep[page 543]{heffernan:resnick:2007}. Note that
$H^{(1)} \equiv H$ according to our definition. Now,
\begin{align}
\mu^{*}([-\infty,\Hinv(\frac 1x)]\times(y,\infty])
 =& \frac{1}{y} \times y \mu^{*}([-\infty,\Hinv(\frac
 1x)]\times(y,\infty])\nonumber \\
 =& \frac 1y \times H(\psi_1(1/y) H\inv(\frac 1x) + \psi_2(1/y)). \label{eqn:muhhinv}
\end{align}
\begin{enumerate}
\item If $\mu^*$ is a product measure then $\mu^* = H \times \nu_1$
where $\nu_1\big((x,\infty]\big) = x^{-1}, x >0$. Similarly,  $\tilde{\mu^*} = \tilde{H} \times \nu_1$. We know that $\mu^*$
being a product measure is equivalent to $\psi_1\equiv
1,\psi_2\equiv 0$. Thus $H^{(c)}\equiv H$ for any $c>0$.
Thus
\begin{align*}
I_{\mu^*} & = \int\limits_1^{\infty}\int\limits_1^{\infty} \mu^*([-\infty,\Hinv(\frac{1}{x})]\times(y,\infty]\big) \frac{dx}{x}\frac{dy}{y}
  = \int\limits_1^{\infty}\int\limits_1^{\infty} \frac{1}{y}
 H\big(\Hinv(\frac{1}{x})\big) \frac{dx}{x}\frac{dy}{y}\\
 & = \int\limits_1^{\infty}\int\limits_1^{\infty} \frac{1}{y}
 \frac 1x  \frac{dx}{x}\frac{dy}{y}
 = \Big(\int\limits_1^{\infty}
 \frac 1 {x^2} dx\Big)^2 =1.
 \end{align*}
Also
 \begin{align*}
I_{\tilde{\mu^*}} & = \int\limits_1^{\infty}\int\limits_1^{\infty} \tilde{\mu^*}([-\infty,\tHinv(\frac{1}{x})]\times(y,\infty]\big) \frac{dx}{x}\frac{dy}{y}
 = \int\limits_1^{\infty}\int\limits_1^{\infty} \mu^*([-\tHinv\big(1\frac{1}{x}\big)]\times(y,\infty]\big) \frac{dx}{x}\frac{dy}{y}\\
 & = \int\limits_1^{\infty}\int\limits_1^{\infty} \frac{1}{y}
 \big(1-H\big(\Hinv(1-\frac{1}{x}\big)\big) \frac{dx}{x}\frac{dy}{y}
 = \Big(\int\limits_1^{\infty}
 \frac 1 {x^2} dx\Big)^2 =1.
 \end{align*}
 \item  Conversely assume that $I_{\mu^*}=I_{\tilde{\mu^*}}=1$.
 We know that $\psi_1(c)= c^{\rho}$ for some $\rho \in \R$. Let us
consider the following cases:
\begin{enumerate}
\item $\rho=0$.
This means $\psi_1 \equiv 1$ and $\psi_2(c) = k \log c$ for some $k \in \R$. We will show that $k$ must be $0$. If $k>0$, then
\begin{align*}
I_{\mu^*} & = \int\limits_1^{\infty}\int\limits_1^{\infty} \mu^*([-\infty,\Hinv(\frac{1}{x})]\times(y,\infty]\big) \frac{dx}{x}\frac{dy}{y}
 = \int\limits_1^{\infty}\int\limits_1^{\infty} \frac{1}{y}
 H\big(\Hinv\big(\frac{1}{x}\big) - k \log y)\big) \frac{dx}{x}\frac{dy}{y}\\
 & < \int\limits_1^{\infty}\int\limits_1^{\infty} \frac{1}{y}
 \frac 1x  \frac{dx}{x}\frac{dy}{y}
                     = \Big(\int\limits_1^{\infty}
 \frac 1 {x^2} dx\Big)^2 =1.
\end{align*}
Similarly, we can show
\begin{align*}
  I_{\mu^*} & \begin{cases}  =1  & \text{if} \quad k  = 0\\
                     >1  & \text{if} \quad k  < 0.
 \end{cases}
 \qquad \text{and} \qquad  I_{\tilde{\mu^*}}  \begin{cases}   > 1  & \text{if} \quad k  > 0\\
  =1  & \text{if} \quad k  = 0\\
                     < 1  & \text{if} \quad k  < 0.
 \end{cases}
 \end{align*}
  Thus for $I_{\mu^*}=I_{\tilde{\mu^*}}=1$ to hold, we must have $k=0$, which implies $\psi_2\equiv 0$ and $\mu^*$ becomes a product measure.\\

\item  $\rho \neq 0$.
We will show that this is not possible under the assumption $I_{\mu^*}=I_{\tilde{\mu^*}}=1$.
For $c>0$,
\[ \psi_1(c)=c^{\rho}, \qquad \psi_2(c) = \frac{k}{\rho} (c^{\rho}-1)\]
for some $k \in \R$. Assume first $\rho > 0$.
 Then $(\frac 1y)^{\rho} \le 1$ for $y \ge 1$.
Therefore, for such $y$,
\begin{align}
& (\frac 1y)^{\rho} H\inv(\frac 1x) + \frac{k}{\rho} ((\frac 1y)^{\rho}-1) \le
H\inv(\frac 1x) \quad
\text{iff} \quad H\inv(\frac 1x) +\frac{k}{\rho} \ge 0 \nonumber \\
& \text{iff} \quad  x \le 1/H(-\frac{k}{\rho})=: \delta, \qquad \delta \ge 1. \label{def:delta}
\end{align}
Denote
\begin{align} \label{def:chi}
\chi(x,y) := H\big((\frac
1y)^{\rho} H\inv(\frac 1x) + \frac{k}{\rho} ((\frac 1y)^{\rho}-1) \big), \qquad x\ge1, y\ge 1.
\end{align}

Since $H$ is non-decreasing
\begin{align}
  \frac{1}{y} \chi(x,y)  & = \mu^*([-\infty,H^\leftarrow (\frac 1x) ]\times(y,\infty])
=\frac 1y H\big((\frac
1y)^{\rho} H\inv(\frac 1x) + \frac{k}{\rho} ((\frac 1y)^{\rho}-1) \big) \nonumber \\ &\le \frac 1y H\big(H\inv(\frac
1x)\big) = \frac 1x \cdot \frac 1y  \qquad \qquad \text{iff $x\le \delta, y \ge 1$.
}\label{eqn:mu*1byxy}
\end{align}
Since $I_{\mu^*}=1$, we have
\begin{align}
&\int\limits_1^{\infty} \int\limits_1^{\infty} \frac{1}{y} \chi(x,y) \frac{dx}{x}\frac{dy}{y} = 1 = \int\limits_1^{\infty} \int\limits_1^{\infty}\frac{1}{x}\frac{1}{y}\frac{dx}{x}\frac{dy}{y}. \label{eqn:chixy}
\end{align}
We claim $1<\delta< \infty$, since  if $\delta$ is either  $1$ or
$\infty$, then
\eqref{eqn:mu*1byxy} and \eqref{eqn:chixy}   imply that
$\chi(x,y)=\frac 1x$ almost everywhere  which means
$$
(\frac
1y)^{\rho} H\inv(\frac 1x) + \frac{k}{\rho} ((\frac 1y)^{\rho}-1) =
\Hinv(\frac 1x)
$$
which is impossible for all $y\geq 1$ when $\rho>0$.

{}From \eqref{eqn:chixy} we have
$$
 \int\limits_1^{\infty} \int\limits_1^{\infty} \frac{1}{y}
 \Big[\frac{1}{x} -\chi(x,y)\Big] \frac{dx}{x}\frac{dy}{y} = 0 ,
$$
that is,
\begin{equation}\label{eqn:chiH}
 \int\limits_1^{\infty} \int\limits_1^{\delta} \frac{1}{y} \Big[\frac{1}{x} -\chi(x,y)\Big]\frac{dx}{x}\frac{dy}{y} = \int\limits_1^{\infty} \int\limits_{\delta}^{\infty} \frac{1}{y} \Big[\chi(x,y)-\frac{1}{x}\Big]\frac{dx}{x}\frac{dy}{y} = \Delta \qquad \mbox{ (say)},
\end{equation}
where the integrands are non-negative on both sides using \eqref{eqn:mu*1byxy}.
Now $I_{\tilde{\mu^*}}=1$ implies that

\begin{align*}
1 &= \int\limits_1^{\infty} \int\limits_1^{\infty}\tilde{\mu^*}([-\infty,\tHinv(\frac 1x)]\times(y,\infty]) \frac{dx}{x}\frac{dy}{y} \\
  & =\int\limits_1^{\infty} \int\limits_1^{\infty} \tilde{\mu^*}([\Hinv(1-\frac 1x),\infty]\times(y,\infty]) \frac{dx}{x}\frac{dy}{y}\\
  & = \int\limits_1^{\infty} \int\limits_1^{\infty} \frac{1}{y}\Big(1-\mu^*([-\infty,\Hinv(1-\frac 1x)]\times(1,\infty])\Big) \frac{dx}{x}\frac{dy}{y}\\
  & = \int\limits_1^{\infty} \int\limits_1^{\infty} \frac{1}{y}\Big(1-H\Big(\psi_1(1/y)\Hinv(1-\frac 1x)+\psi_2(1/y)\Big)\Big) \frac{dx}{x}\frac{dy}{y}\\
  & = \int\limits_1^{\infty} \int\limits_1^{\infty} \frac{1}{y}\Big[1-\chi(\frac{x}{x-1},y)\Big] \frac{dx}{x}\frac{dy}{y}.
  \intertext{Use the transformation $z=\frac{x}{x-1}$ and the above equation becomes}
  1& = \int\limits_1^{\infty} \int\limits_1^{\infty} \frac{1}{y} \frac{1}{z-1}\Big[1-\chi(z,y)\Big] \frac{dz}{z}\frac{dy}{y} \\
   & = \int\limits_1^{\infty} \int\limits_1^{\infty} \frac{1}{y} \frac{1}{z-1}\Big[\frac{1}{z}-\chi(z,y)\Big] \frac{dz}{z}\frac{dy}{y} + \int\limits_1^{\infty} \int\limits_1^{\infty} \frac{1}{y} \frac{1}{z-1}\Big[1-\frac{1}{z}] \frac{dz}{z}\frac{dy}{y} \\
   & =  \int\limits_1^{\infty} \int\limits_1^{\infty}\frac{1}{z-1} \frac{1}{y} \Big[\frac{1}{z}-\chi(z,y)\Big] \frac{dz}{z}\frac{dy}{y} + 1.
\end{align*}
Therefore we have
\begin{align}
&\int\limits_1^{\infty} \int\limits_1^{\infty} \frac{1}{x-1}\frac{1}{y} \Big[\frac{1}{x}-\chi(x,y)\Big] \frac{dx}{x}\frac{dy}{y} =0. \nonumber
\intertext{Since $\chi(x,y) \le \frac{1}{x}$ if and only if $x \le \delta$ from \eqref{eqn:mu*1byxy}, we have}
&\int\limits_1^{\infty} \int\limits_1^{\delta} \frac{1}{x-1}\frac{1}{y} \Big[\frac{1}{x} -\chi(x,y)\Big]\frac{dx}{x}\frac{dy}{y} = \int\limits_1^{\infty} \int\limits_{\delta}^{\infty} \frac{1}{x-1} \frac{1}{y} \Big[\chi(x,y)-\frac{1}{x}\Big]\frac{dx}{x}\frac{dy}{y} \label{eqn:chiHxminus1}
\intertext{where the integrands on both sides are non-negative. But referring to \eqref{eqn:chiH} we have}
& \int\limits_1^{\infty} \int\limits_1^{\delta} \frac{1}{x-1}\frac{1}{y} \Big[\frac{1}{x} -\chi(x,y)\Big]\frac{dx}{x}\frac{dy}{y} \ge \frac{1}{\delta-1} \int\limits_1^{\infty} \int\limits_1^{\delta} \frac{1}{y} \Big[\frac{1}{x} -\chi(x,y)\Big]\frac{dx}{x}\frac{dy}{y} = \frac{\Delta}{\delta-1} \label{eqn:contra1}
\intertext{with equality holding only if the integrand is $0$ almost everywhere. Similarly we have }
& \int\limits_1^{\infty} \int\limits_{\delta}^{\infty} \frac{1}{x-1}\frac{1}{y} \Big[\chi(x,y)- \frac{1}{x} \Big]\frac{dx}{x}\frac{dy}{y} \le  \frac{1}{\delta-1} \int\limits_1^{\infty} \int\limits_{\delta}^{\infty} \frac{1}{y} \Big[\chi(x,y)-\frac{1}{x}\Big]\frac{dx}{x}\frac{dy}{y} = \frac{\Delta}{\delta-1} \label{eqn:contra2}
\end{align}
with equality holding only if the integrand is $0$ almost everywhere. The integrand cannot be $0$ since it will imply  $\chi(x,y) = \frac{1}{x}$ almost everywhere meaning $\rho=0$. But our assumption is $\rho>0$. Thus with strict inequality holding for both \eqref{eqn:contra1} and \eqref{eqn:contra2} we have a contradiction in equation \eqref{eqn:chiHxminus1}. Thus we cannot have $\rho >0$.\\

The case with $\rho<0$ can be proved similarly.
Hence the result.
\end{enumerate}\end{enumerate}
\end{proof}
This corollary provides a detection technique for the limit measure
being a product measure. Given a sample of size $n$, we plot
$\Hillish$ for values of $k$ and then try to see whether it
stabilizes close to $1$ or not. If the statistic is close to another value, this is evidence that the model is applicable but the limit measure is not product.

 \subsection{The Pickandsish statistic, $\Pick(p)$}
Another way to check the suitability of the CEV assumption and to detect a product measure
in the limit is to use the Pickandsish statistic which is based on ratios of differences of ordered concomitants.  The statistic is patterned on the Pickands estimate
for the parameter of an extreme value distribution
(\citet{pickands:1975},
\citet[page 83]{dehaan:ferreira:2006},
\citet[page 93]{resnickbook:2007}). For a fixed $k<n$, recall that $X_{1:k}^* \le
\ldots \le
 X_{k:k}^*$ are the order statistics in increasing order from
 $X_1^*,X_2^*,\ldots,X_k^*$, the concomitants of $Y_{(1)}\ge \ldots
 \ge Y_{(k)}$, the order statistics in decreasing order from
 $Y_1,Y_2,\ldots,Y_n$. For notational convenience for $s\le t$ write
 $X_{s:t}:=X_{\lceil s \rceil:\lceil t \rceil}$. Now define the
 Pickandsish statistic  for $0<p<1$,
 \begin{align}\label{def:pick}
 \Pick(p) := \frac{X_{pk:k}^*-X_{pk/2:k/2}^*}{X_{pk:k}^*-X_{pk/2:k}^*}.
 \end{align}

 \begin{Proposition} \label{prop:convpick}
 Suppose $(X_1,Y_1),\ldots,(X_n,Y_n)$ follows a CEV model. Let  $0<p<1$,
  Then, as $k,n\to \infty$ with $k/n \to
 0$, we have
 \begin{align}
 \Pick(p) & \cinP \frac{H\inv(p)(1-2^{\rho}) -
 \psi_2(2)}{H\inv(p)-H\inv(p/2)}, \label{eqn:rplim}
 \end{align}
provided $H\inv(p)-H\inv(p/2) \neq 0$. Here $\psi_1$ and $\psi_2$ are  defined in
 \eqref{eqn:condpsi1psi2} and $\rho = \log(\psi_1(c))/\log c$.
  \end{Proposition}

 \begin{proof}
Since $H_n $ in \eqref{eqn:nonumber} is a probability distribution
converging  to the limit $H$,
we may  invert the convergence and obtain
\citep[Proposition 2.2, page 20]{resnickbook:2007},
$$H_n\inv (z)\cinP H\inv(z)$$
 for $0 < z < 1$ for which $\Hinv$ is continuous. The convergence of
 $H_n^\leftarrow(\cdot)$ translates to
\begin{align}
H_n\inv (z) & = \inf \{ u \in \R: H_n(u)\ge z \}
            = \inf \{u \in \R: \sum\limits_{i=1}^{k} \epsilon_{\big(\frac{X_i^*- \beta
            (n/k)}{\alpha(n/k)}\big)}[-\infty,u] \ge kz \} \nonumber\\
 & = \frac{X^*_{\lceil kz \rceil:k} - \beta
            (n/k)}{\alpha(n/k)}  \Rightarrow H\inv(z)
            \label{eqn:xquan}
\end{align}
where $X^*_{1:k} \le \ldots X^*_{k:k}$ are the increasing
            order statistics of the concomitants $X^*_1,\ldots
            X^*_k$.

 From \eqref{eqn:xquan}, we have, for $0 < p \le 1 $,
 if $k,n\to \infty$ and $k/n\to 0$,
 \begin{align}
 &\frac{X^*_{ pk :k} - \beta
            (n/k)}{\alpha(n/k)}  \cinP H\inv(p), \label{eqn:pick1}\\
             &\frac{X^*_{ (p/2)k :k} - \beta
            (n/k)}{\alpha(n/k)}  \cinP H\inv(p/2), \label{eqn:pick2}\\
&\frac{X^*_{ pk/2 :k/2} - \beta
            (2n/k)}{\alpha(2n/k)}  \cinP H\inv(p).\label{eqn:pick3}
 \end{align}
 Also recall from \eqref{eqn:condpsi1psi2} that
 \begin{align*}
&\lim\limits_{t \to \infty} \frac{\alpha(tc)}{\alpha(t)}  =
\psi_1(c)=c^{\rho}, & &\lim\limits_{t \to \infty}
\frac{\beta(tc)-\beta(t)}{\alpha(t)}  = \psi_2(c) ,
\end{align*}
where $\psi_2$ can be either 0 or $\psi_2(c) = D
\frac{c^{\rho}-1}{\rho}$ for some $D \neq 0$ and $\rho \in \R$. Now
note that using Slutsky's theorem we have
\begin{align*}
\frac{X^*_{ pk :k}-X^*_{ pk/2 :k/2}}{\alpha(n/k)} & = \frac{X^*_{ pk
:k}-\beta(n/k)}{\alpha(n/k)} - \frac{X^*_{ pk/2 :k/2} -
\beta(2n/k)}{\alpha(2n/k)}\times\frac{\alpha(2n/k)}{\alpha(n/k)} -
\frac{\beta(2n/k)-\beta(n/k)}{\alpha(n/k)}\\
& \cinP H\inv(p) - H\inv(p)2^{\rho} - \psi_2(2), \qquad
 \intertext{and also,} \frac{X^*_{ pk :k}-X^*_{
(p/2)k :k}}{\alpha(n/k)} & = \frac{X^*_{ pk
:k}-\beta(n/k)}{\alpha(n/k)} - \frac{X^*_{ (p/2)k :k} -
\beta(n/k)}{\alpha(n/k)} \\
& \cinP H\inv(p) - H\inv(p/2). \intertext{Since $H\inv(p) -
H\inv(p/2) \neq 0$, another use of Slutsky gives us} \Pick(p) & =
\frac{(X^*_{ pk :k}-X^*_{ pk/2 :k/2})\alpha(n/k)}{(X^*_{ pk
:k}-X^*_{ (p/2)k :k})\alpha(n/k)}  \cinP \frac{H\inv(p)(1-2^{\rho})
-
 \psi_2(2)}{H\inv(p)-H\inv(p/2)}.
\end{align*}
 \end{proof}

\begin{Corollary}\label{cor:whenPickImpliesProd}
Suppose there exists $0<p_1 < p_2<1$ such that $H\inv(p_1)< H\inv(p_2)$, and for $i=1,2$,
$H\inv(p_i)-H\inv(p_i/2)\neq 0$. Then under the conditions of
Proposition \ref{prop:convpick}, $\mu^*$ is a product measure if and
only if $$\Pick(p_i) \cinP 0, \quad i=1,2.$$
\end{Corollary}

 \begin{proof}
{$1^o.$} Assume that $\mu^*$ is a product measure. Then
$(\psi_1,\psi_2)\equiv (1,0)$, i.e., $\rho=0$ and $\psi_2\equiv0$. Hence
$$ H\inv(p)(1-2^{\rho}) - \psi_2(2) = H\inv(p) (1-1) - 0=0.$$
Therefore, provided $0<p<1$ and $H\inv(p)-H\inv(p/2)\neq 0$, Proposition
\ref{prop:convpick} implies $\Pick(p) \cinP 0$.

$2^o.$ Conversely, suppose $p_1 < p_2$ and $\Pick(p_i) \cinP 0$, $i=1,2$. Hence
\begin{align}
 H\inv(p_1)(1-2^{\rho}) - \psi_2(2) & = 0 \label{eqn:pickzero}
\end{align}
\begin{enumerate}
\item Suppose $\rho = 0$ which means $\psi_1\equiv 1$. Then
\eqref{eqn:pickzero} implies $\psi_2(2)=0$ which implies
$\psi_2\equiv 0$. This means $\mu^*$ is a product measure.
\item Suppose $\rho \neq 0$ and $\psi_2 \equiv 0$. Then
\eqref{eqn:pickzero} implies $H\inv(p_i)(1-2^{\rho})=0, i=1,2$. This
implies $H\inv(p_i)=0, i=1,2$, a contradiction to $H\inv(p_1) <
H\inv(p_2)$. So this supposition is not possible.
\item Suppose $\rho \neq 0$ and $\psi_2(c) = D\frac{c^{\rho}-1}{\rho}$
  for $D \neq 0$. Then
\eqref{eqn:pickzero} implies
$(H\inv(p_i)+\frac{D}{\rho})(1-2^{\rho})=0, i=1,2$. This means
$H\inv(p_i)= - \frac{D}{\rho}, i=1,2$, a contradiction to
$H\inv(p_1) < H\inv(p_2)$. So this supposition is not possible.
\end{enumerate}
 Hence we have that $\mu^*$ is a product measure if for $p_1 < p_2$ we have $\Pick (p_i) \cinP 0, i=1,2$.
  \end{proof}

  \subsection{Kendall's Tau  $\rho_{\tau}(k,n)$}
  Classically,  Kendall's tau statistic (\cite{mcneil:frey:embrechts:2005})  is used to measure the strength of association between two rankings. We use  a slightly modified version of the statistic using data pertaining to the $k$ maximum  $Y$-values: $Y_{(1)} \ge \ldots \ge Y_{(k)}$, their concomitants $X_1^*,\ldots,X_k^*$ and the ranks $R_1,\dots,R_k$ of $X_1^*,\dots,X_k^*$. The Kendall's tau statistic is
  \begin{align}\label{def:kendall}
 \rho_{\tau}(k,n) &:=  \frac {4}{k(k-1)} {\sum\limits_{1\le i < j \le k} \bone_{\{R_i <R_j\}} }-1.
 \end{align}
  This statistic can also be used to show the appropriateness
 of the CEV model and to decide if the limit measure is a product. We
 show that under the CEV model $\rho_{\tau}(k,n)$ as defined in
 \eqref{def:kendall} converges in probability to a limiting constant  and when
 the CEV model holds with a product measure, the limit is 0.

   First we prove a lemma on copulas in $[0,1]^2$ which leads to proving convergence for the statistic $\rho_{\tau}(k,n)$. Recall that a two dimensional copula is any distribution function defined on $[0,1]^2$ with uniform marginals (\cite{mcneil:frey:embrechts:2005}).
 \begin{Lemma} \label{lem:copulas}
Suppose $\{C_{\infty}, C_n \;n \ge 1\} $ are copulas on $[0,1]^2$,
$C_{\infty}$ is continuous and $C_n \Rightarrow C_{\infty}$. Then
\begin{align} \label{eqn:copn}
\int\limits_{[0,1]^2} C_n(u-,v-)dC_n(u,v) & \to
\int\limits_{[0,1]^2} C_{\infty}(u,v)dC_{\infty}(u,v),\qquad (n\to\infty).
\end{align}
 \end{Lemma}
\begin{proof}
Since $C_n \Rightarrow C_{\infty}$, the convergence is uniform, that
is, we have $$ ||C_n-C_{\infty}||:=\sup\limits_{(u,v)\in [0,1]^2}
|C_n(u,v) - C_{\infty}(u,v)| \to 0.$$ Therefore
\begin{align*}
& \Big| \int\limits_{[0,1]^2} C_n(u-,v-)dC_n(u,v) -
\int\limits_{[0,1]^2} C_{\infty}(u,v)dC_{\infty}(u,v) \Big| \\
\le &   \int\limits_{[0,1]^2} | C_n(u-,v-) - C_{\infty}(u,v)| dC_n(u,v) +
\Big| \int\limits_{[0,1]^2} C_{\infty}(u,v)dC_n(u,v) -
\int\limits_{[0,1]^2} C_{\infty}(u,v)dC_{\infty}(u,v) \Big|\\
\le & ||C_n-C_{\infty}|| + \Big| \int\limits_{[0,1]^2}
C_{\infty}(u,v)dC_n(u,v) -
\int\limits_{[0,1]^2} C_{\infty}(u,v)dC_{\infty}(u,v) \Big|\\
& \to 0.
\end{align*}
\end{proof}

\begin{Remark}\label{rem:randomVersion}
From Lemma \ref{lem:copulas} we get that if
$\{C_n;n\ge1\}$ are random probability measures and $C_\infty$ is continuous, then
\begin{align} \label{eqn:copnrnd}
\int\limits_{[0,1]^2} C_n(u-,v-)dC_n(u,v) & \cinP
\int\limits_{[0,1]^2} C_{\infty}(u,v)dC_{\infty}(u,v)
\end{align}
\end{Remark}

Define the following copulas on $[0,1]^2$:
 \begin{align}
 C_{\mu^*_n}(x,y) &: = \frac 1k \sum\limits_{i=1}^k
 \epsilon_{\frac{R_i^k}{k},\frac{i}{k}}([0,x]\times[0,y]), \qquad (x,y)\in [0,1]^2 \label{def:cmun}\\
 C_{\mu^*}(x,y) &: = \mu^*([\infty,\Hinv(x)] \times[y^{-1},\infty]). \label{def:cmu}
 \end{align}

 \begin{Proposition}\label{prop:kendallconv}
Suppose $(X_1,Y_1),(X_2,Y_2),\ldots,(X_n,Y_n)$ are i.i.d. observations
from a CEV model which follows \eqref{eqn:doa}-\eqref{eqn:Hispdf} and
suppose $H$ defined in \eqref{eqn:Hispdf}
is continuous. Assume that $k=k(n) \to \infty, n
\to \infty$ and $k/n\to 0$. Then
\begin{align} \label{eqn:kendalllim}
 \Kendall  \cinP 4\int\limits_{[0,1]^2} C_{\mu^*}(x,y)dC_{\mu^*}(x,y)-1 = : J_{\mu^*}.
\end{align}
If $\mu^*$ is a product measure, $J_{\mu^*} =0$
\end{Proposition}

\begin{proof}
Proposition \ref{prop:empconv} implies that as $k,n \to \infty$ with
$k/n \to 0$, for $0\le x \le 1$, $z \ge 1$
$$
 \frac 1k \sum\limits_{i=1}^k \epsilon_{\big(\frac{R_i}k , \frac {k+1}i
\big)} ([0,x] \times[z,\infty])   \Rightarrow \mu^*([-\infty,H\inv
(x)] \times(z,\infty]).
$$
Therefore, for $0\le x,y \le 1$,
\begin{align*}
C_{\mu^*_n}(x,y) &:= \frac 1k \sum\limits_{i=1}^k \epsilon_{\big(\frac{R_i}k , \frac
 ik
\big)} ([0,x] \times[0,y])
 = \frac 1k \sum\limits_{i=1}^k \epsilon_{\big(\frac{R_i}k , \frac
 ik
\big)} ([0,x] \times[0,y)) + o_P(1)\\  & \Rightarrow \mu^*([-\infty,H\inv (x)]
\times(y^{-1},\infty]) = C_{\mu^*}(x,y) .
\end{align*}
since H is continuous, and replacing $k+1$ by $k$
does not matter in the limit. This shows that $C_{\mu_n} \Rightarrow
C_{\mu^*}$.  {}From
 Lemma \ref{lem:copulas}  and Remark \ref{rem:randomVersion}
we have
$$
 S_n^*  : = \int\limits_{0}^{1}\int\limits_{0}^{1}
 C_{\mu^*_n}(x-,y-) dC_{\mu^*_n}(x,y)
 \Rightarrow \int\limits_{0}^{1}\int\limits_{0}^{1}
 C_{\mu^*}(x,y) dC_{\mu^*}(x,y).
$$
 Now note that
\begin{align*}
S_n^* & = \int\limits_{0}^{1}\int\limits_{0}^{1}
 C_{\mu^*_n}(x-,y-) dC_{\mu^*_n}(x,y)
 =   \frac 1k \sum\limits_{i=1}^k
  C_{\mu^*_n}(\frac{R_i}{k}-,\frac{i}{k}-) \\
 & = \frac1{k^2}\sum\limits_{i=1}^k \sum\limits_{l=1}^k
 \epsilon_{\{\frac{R_l}{k}, \frac lk\}}([0,\frac{R_i}{k})\times[0,\frac{i}{k}))
 = \frac1{k^2}\sum\limits_{1 \le l < i \le k} \bone_{\{R_l <
 R_i\}} \\
 &= \frac {k(k-1)}{4k^2} \rho_{\tau}(k,n) - \frac{1}{k}.
\end{align*}
 Hence we have as $k,n \to \infty$ with $k/n \to 0$,
$$
 \rho_{\tau}(k,n)= \frac{k}{k-1}(4S_n^*-1)+ \frac{1}{k-1}
 \Rightarrow 4 \int\limits_{0}^{1}\int\limits_{0}^{1}
 C_{\mu^*}(x,y) dC_{\mu^*}(x,y)-1 =: J_{\mu^*}.
$$

If $\mu^*$ is a product, for $0\le x,y \le 1$,
$$
 C_{\mu^*}(x,y) : = \mu^*([-\infty,\Hinv(x)] \times[y^{-1},\infty])
                 = y \times H(\Hinv(x)) = xy.
$$
 Hence
$$
\int\limits_{0}^{1}\int\limits_{0}^{1}
 C_{\mu^*}(x,y) dC_{\mu^*}(x,y)  =
 \int\limits_{0}^{1}\int\limits_{0}^{1} xy dx dy = \frac 14
$$
and the result follows.
\end{proof}

Proposition \ref{prop:kendallconv} would detect that a limit is not a
product if the statistics stabilizes at a non-zero value.
We have not been able to prove a limit of 0 implies a product measure
and doubt the truth of this statement.

\begin{Remark}
 The three statistics provided above each have their own advantages and disadvantages.
 \begin{itemize}
 \item They are not hard to calculate.
 \item For the CEV model we have shown that all these statistics stabilize as $k,n \to \infty$ with $k/n \to 0$.
\item The rank-based statistics Hillish and Kendall's tau are smooth in nature as the rank transform removes the extremely high or low values.
\item The disadvantage of the Pickandsish statistic  is that its plot lacks smoothness  and exhibits erratic behavior for small data sets.
\item Obtaining distributional properties for these statistics would require further limit conditions on the variables, presumably some form of second order behavior.
\end{itemize}
\end{Remark}

\section{Examples and applications}\label{sec:data}
In this section we apply the three estimators proposed in Section
\ref{sec:product} to  data sets and judge their performances in the
various cases. First we deal with simulated data from
specific models discussed in \cite{das:resnick:2008a}. Then we apply
our techniques to Internet traffic data.

\subsection{Simulation from known CEV limit models}\
\begin{Example} \label{ex:ex1}
Let $X$ and $Y$ be independent random variables with $X \sim N(0,1)$
and $Y \sim Pareto(1)$. Then the following convergence holds in
$\mathbb{M}_{+}([-\infty,\infty]\times(0,\infty])$ (actually it is an equality)
\begin{align*}
  & t\P \Bigl[\Bigl( X,\frac{Y}{t}\Bigr) \in
 [-\infty,x]\times(y,\infty]\Bigr] = \Phi(x) y^{-1}, \quad -\infty < x < \infty , \,y
 \ge 1
\end{align*}
where $\Phi$ denotes the standard normal distribution function. We
have a CEV model here with $\alpha \equiv 1,\, \beta \equiv 0, \,
a(t)=t, \, b \equiv 0$. The limit measure is a product. Hence,
theoretically
$$
\Hillish \cinP
1,\qquad  \rho_{\tau}(k,n) \cinP 0, \qquad \Pick(p) \cinP 0, \qquad
0<p<1.
$$
We simulate
a sample of size $n=1000$ and plot the above estimators for $1\le k
\le n$. For the Pickandsish statistic we have chosen $p=0.5$.
\begin{figure}[h]
\begin{centering}
\includegraphics[width=150mm]{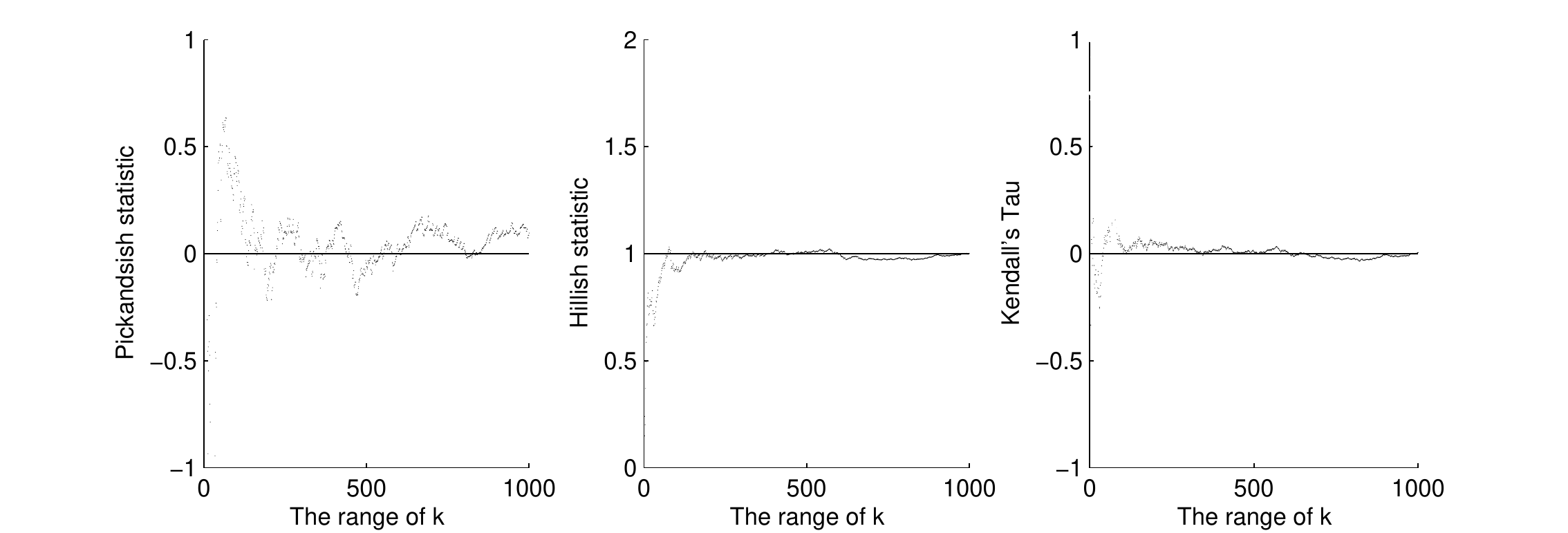}
\end{centering}
\caption{  $\Pick(0.5)$, $\Hillish$ and $\rho_{\tau}(k,n)$ plotted for Example  \ref{ex:ex1} }
\end{figure}
The simulated data supports the theoretical results stated.
\end{Example}

\begin{Example} \label{ex:ex2}
Let $X$ and $Z$ be independent Pareto random variables where $X \sim
Pareto(\rho)$ and $Z \sim Pareto(1-\rho)$ with $0 < \rho <1$. Define $
Y = X \wedge Z. $ Then we can check that the following holds in
$\mathbb{M}_{+}([0,\infty]\times(0,\infty])$: For $x \ge y >0$  and
  $t$ large,
$$
 t\P \Bigl[\Bigl( \frac{X}{t},\frac{Y}{t}\Bigr) \in
  [0,x]\times(y,\infty]\Bigr]  =
  \frac{1}{y^{1-\rho}}\Bigl(\frac{1}{y^{\rho}} -
  \frac{1}{x^{\rho}}\Bigr) = \frac{1}{y}\Bigl(1 -
  \frac{y^{\rho}}{x^{\rho}}\Bigr) =:
  \mu^{*}([0,x]\times(y,\infty]).
$$
Theoretically the values of the limits of $\Hillish$ and $\Pick(p)$ are as
follows.
\begin{align*}
\Hillish & \cinP \int\limits_1^{\infty}\int\limits_1^{\infty}
\mu^*([0,\Hinv(\frac{1}{x})]\times(y,\infty]) \frac{dx}{x} \frac{dy}{y}
 = \int\limits_{x=1}^{\infty}\int\limits_{y=1}^{(\frac{x}{x-1})^{1/\rho}}
\frac 1y \Big( 1 -
\frac {y^{\rho}} {\frac{x}{x-1}}\Big) \frac{dx}{x}  \frac{dy}{y}\\
& = \int\limits_{x=1}^{\infty} \Big[1 - \frac{1}{1-\rho}\frac{x-1}{x} + \frac{\rho}{1-\rho} \big(\frac{x-1}{x}\big)^{1/\rho}\Big] \frac{dx}{x}\\
& = \frac{\rho}{1-\rho} \int\limits_{x=1}^{\infty} \sum\limits_{k=2}^{\infty} \binom{1/\rho}{k} \Big(\frac{1}{x}\Big)^{k+1} dx
 = \frac{\rho}{1-\rho}\sum\limits_{k=2}^{\infty} \binom{1/\rho}{k}
 \frac{1}{k}.
\end{align*}
Now for $0<p<1$ we have,
\begin{align*} \Pick(p)
& \cinP \frac{H\inv(p)(1-2^{\rho}) -
 \psi_2(2)}{H\inv(p)-H\inv(p/2)}
 = \frac{1-2^{\rho}}{1- \big(\frac{1-p}{1-p/2}\big)^{1/\rho}}.
\end{align*}
For calculating the Kendall's tau statistics observe that from definition we have:
 \begin{align*}
  C_{\mu^*}(x,y)  & = y \Big(1- \frac{1}{y^\rho}(1-x)\Big) = y - y^{1-\rho}(1-x),\\
  dC_{\mu^*}(x,y) & = (1-\rho)y^{-\rho} dx dy.
 \intertext{Hence we have}
 \int\limits_{0}^{1}\int\limits_{0}^{1}
 C_{\mu^*}(x,y) dC_{\mu^*}(x,y) & = (1-\rho)\int\limits_{0}^{1}\int\limits_{0}^{1} ( y - y^{1-\rho}(1-x)) y^{-\rho} dx dy\\
  & = \frac{1-\rho}{2-\rho} - \frac14.
  \intertext{Therefore}
  \rho_{\tau}(k,n) & \cinP 4\int\limits_{0}^{1}\int\limits_{0}^{1}
 C_{\mu^*}(x,y) dC_{\mu^*}(x,y) -1 = \frac{4(1-\rho)}{2-\rho} -2 = -\frac{2\rho}{2-\rho}.
 \end{align*}
For $\rho=0.5$ and $p=0.5$, theoretically we have
$$
\Hillish \cinP
0.5,\qquad  \rho_{\tau}(k,n) \cinP -0.67, \qquad \Pick(p) \cinP -0.75.
$$
 We simulate a sample of size $n=1000$ with $\rho=0.5$ and plot the three statistics
for $1\le k \le n$. For the Pickandsish statistic we have chosen $p=0.5$.
\begin{figure}[h]
\begin{centering}
\includegraphics[width=150mm]{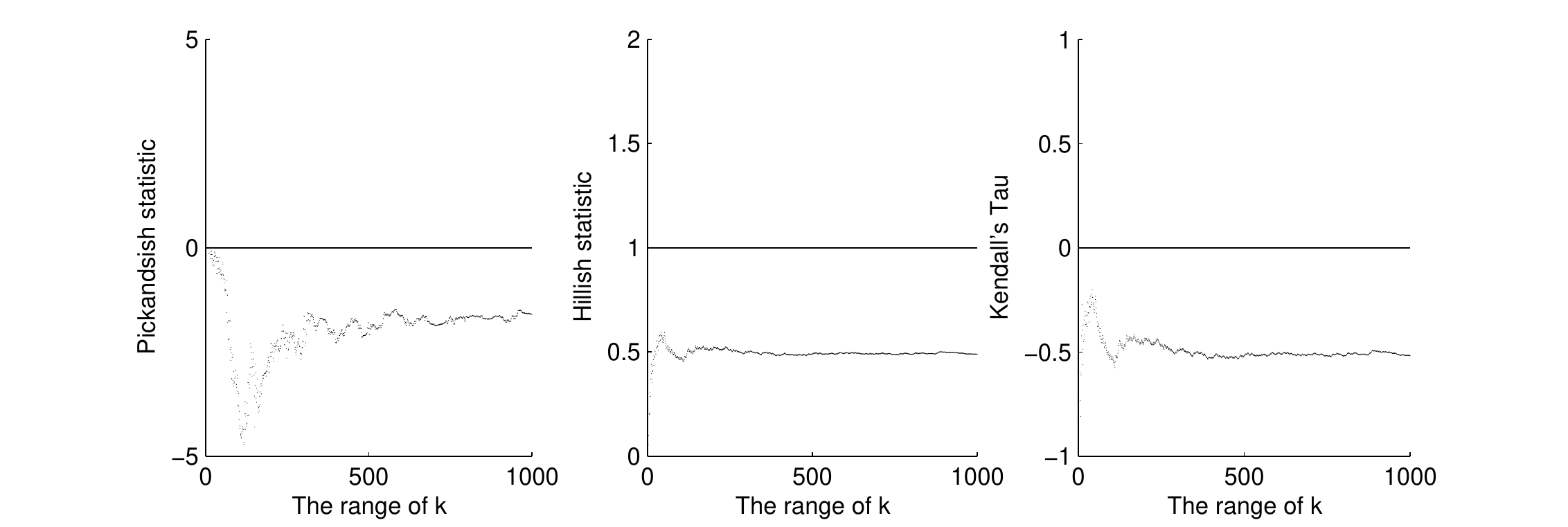}
\end{centering}
\caption{  $\Pick(0.5)$, $\Hillish$ and $\rho_{\tau}(k,n)$ plotted for Example  \ref{ex:ex2}}
\end{figure}
The graphs are consistent with the obtained theoretical limits.
\end{Example}

\subsection{Internet traffic data}\label{subsec:Internet}
Internet traffic data has often provided scope for heavy-tail
modeling. Variables such as file size, transmission duration and
session length have been observed to be heavy-tailed
\citep{maulik:resnick:rootzen:2002,resnick:2003,sarvotham:riedi:baranuik:2005}.

   We study a particular data set of GPS-synchronized traces  that
   were recorded at the University of Auckland
   \emph{http://wand.cs.waikato.ac.nz/wits}. The raw data contains
   measurements on packet size, arrival time, source and destination
   IP, port number, Internet protocol, etc. We consider traces
   corresponding  exclusively to incoming TCP traffic sent on December
   8, 1999, between 3  and 4 p.m.  The packets were clustered into
   end-to-end (e2e)
   sessions  which are  clusters of packets with the same
   source and destination IP address such that the delay between
   arrival of two successive packets in a session is at most two
   seconds. We observe three variables $ \{(S_i,L_i,R_i): 1 \le i \le
   54353\}$:

 \begin{align*}
 S_i & = \text{size or number of bytes transmitted in a session} ,\\
 L_i & = \text{duration or length of the session},\\
 R_i & = \frac{S_i}{L_i} ~~\text{ or average transfer rate associated with a session.}
 \end{align*}
The data have been downloaded and processed into sessions by Luis
Lopez Oliveros, Cornell University.
\begin{figure}[h]
\begin{centering}
\includegraphics[width=170mm,height =60mm]{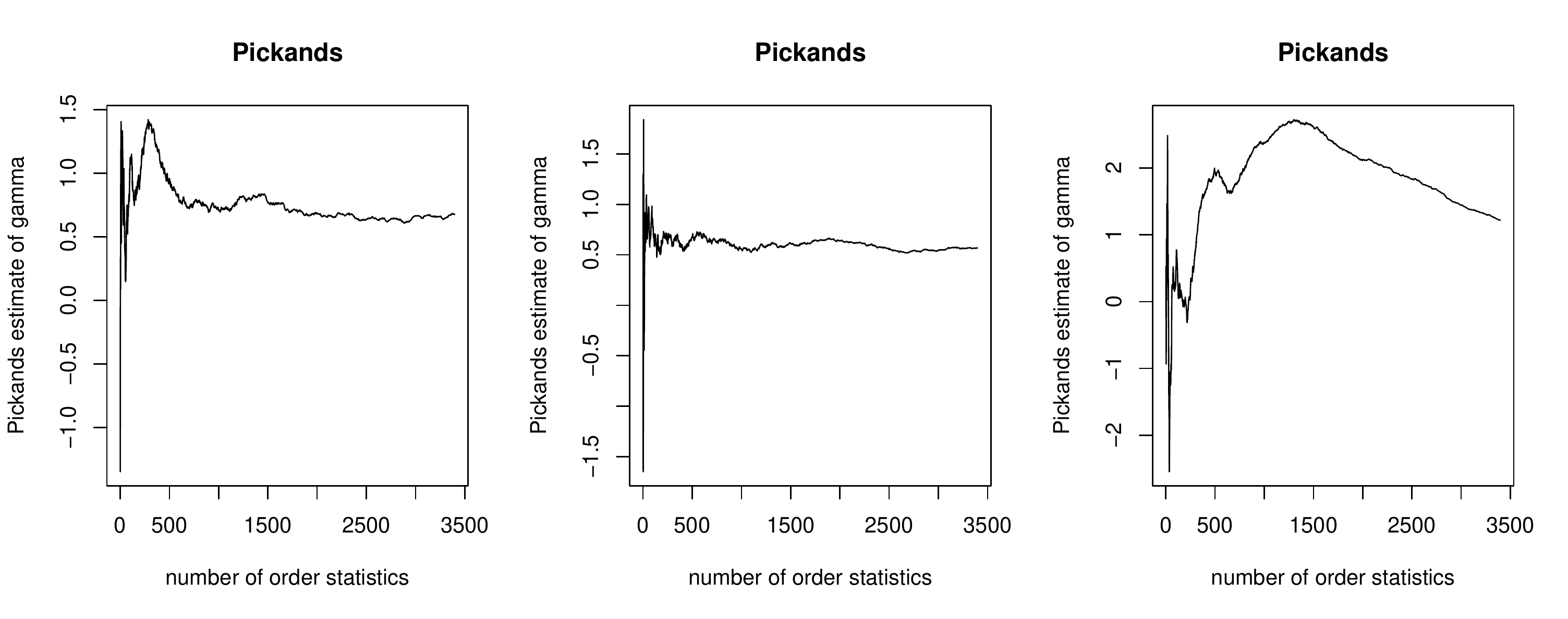}
\hfill
\includegraphics[width=170mm,height =60mm]{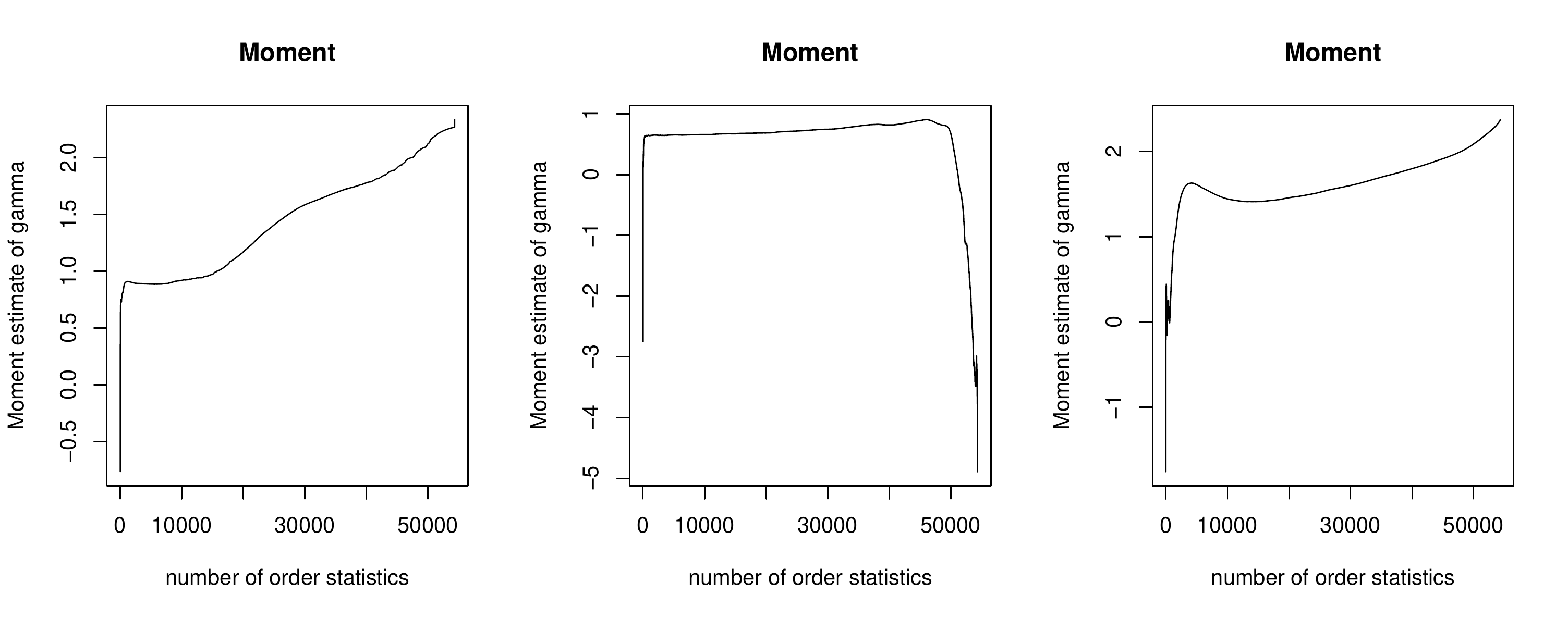}
\end{centering}
\caption{Pickands plot and moment estimate plot of the EV parameter for  size, duration
  and transfer rate.} \label{fig:momall}
\end{figure}

First we check whether the individual variables are heavy-tailed or
not. The Pickands estimator and moment estimators are weakly
consistent for the extreme value parameter $\gamma$
(\cite{dehaan:ferreira:2006}) when the distribution of the variable
under consideration is in $D(G_{\gamma})$
as in \eqref{eqn:Ggamma}. We plot these estimators over $1\le k\le n$ and
observe whether they stabilize over an interval. The Pickands plot
indicates that the Pickands estimates of the extreme value parameter
are stable for  size and duration but not  for the transfer
rate. The moment plot on the other hand shows that the moment estimate
of the extreme-value parameter stabilizes for duration but does not do
that clearly for either  size or transfer rate. Recall that the
CEV model is applicable if either of the variables is in the domain of
attraction of an extreme-value distribution. Clearly there is an
indication that transfer rate might not be in an extreme value
domain.

Now we turn to the three statistics we have devised in this
paper first to detect whether we have a CEV model and  then to check
whether the limit measure is a product. First we consider the pair
$(R,S)$ assuming the distribution of $S$ is in
$D(G_{\gamma})$ for some $\gamma \in \R$. Then
we consider the pair $(R,L)$ assuming  the distribution of $L$ is in $
D(G_{\lambda})$ for some $\lambda \in \R$.

\begin{figure}[h]
\begin{centering}
\includegraphics[width=180mm]{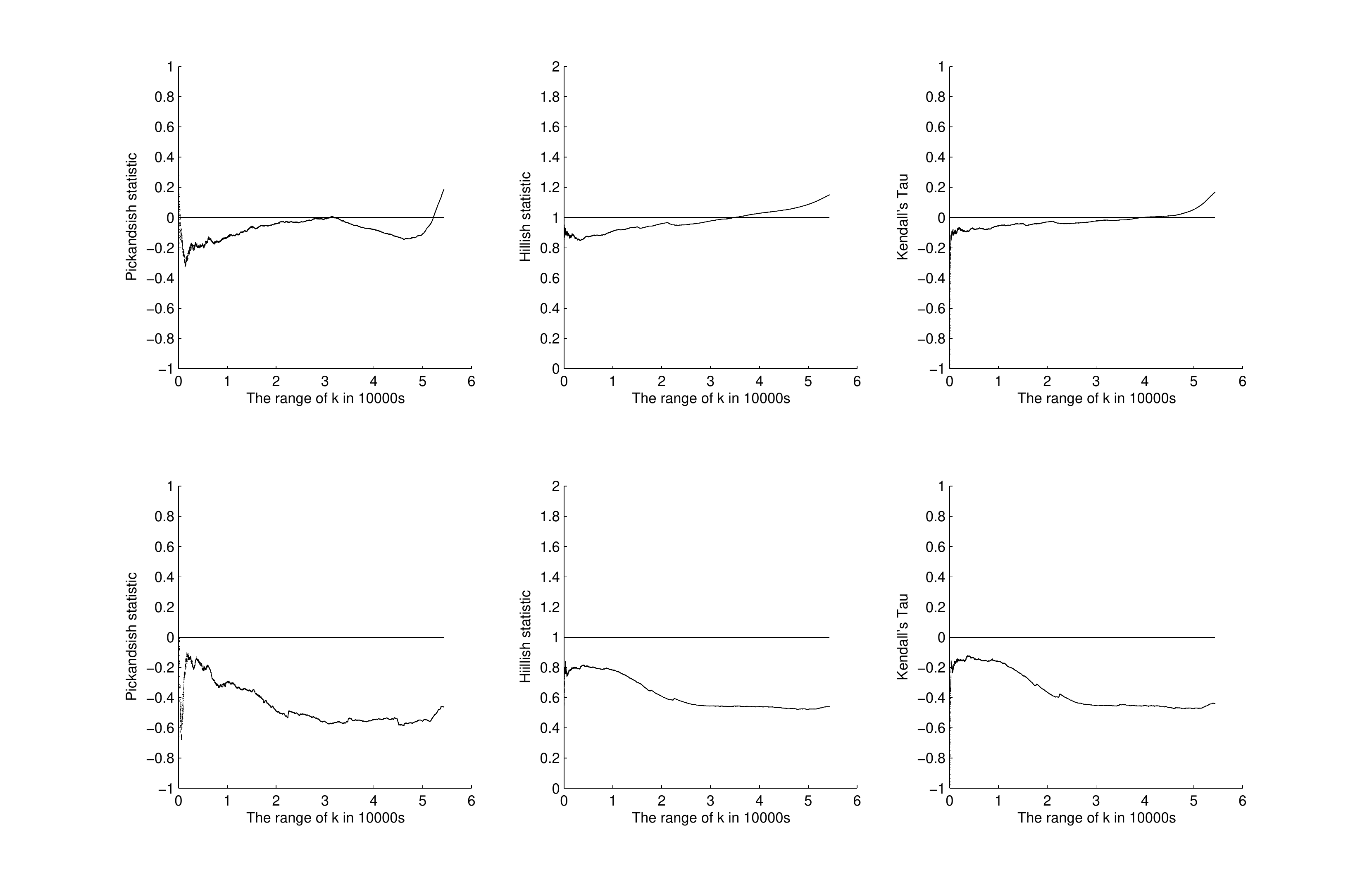}
\end{centering}
\caption{First row: Three statistics for $R_i$ vs. $S_i$;
Second  row: Three statistics for $R_i$ vs. $L_i$.}\label{fig:e2esesalln}
\end{figure}

Observe from Figure \ref{fig:e2esesalln} that neither of the three
statistics stabilizes for the observations  $ (R,S)$. Hence a CEV
model might not be the right model to apply. On the other hand for
$(R,L)$, all the statistics stabilize at some point. But it is clear
they are not stabilizing at a point to indicate product measure. Hence
we have evidence to model (\emph{transfer rate, duration}) as a CEV
model with a non-product limit. Note that this also indicates that we
should be able to standardize to regular variation on
$[0,\infty]\times(0,\infty]$.

\section{Conclusion}
The CEV model is intended to provide us with a deeper understanding of
multivariate distributions which have some components in an
extreme-value domain.  In our discussion,  we have provided
statistics to detect the CEV model in a bivariate set up. These three
statistics perform differently for different data sets as we have
noted in our examples. A further step would be to find asymptotic
distributions for these statistics.
On another direction, it would be nice to obtain statistics for detection of conditional models in a multivariate set up of dimension more than two.
\bibliography{c:/bikram/Bibfiles/bibfile}

\end{document}